\documentclass[11pt, reqno]{amsart}

\usepackage[T1]{fontenc}

\usepackage{amsmath, amssymb, amsfonts, amsthm}
\usepackage{mathtools}      
\usepackage{physics}        
\usepackage{esint}
\usepackage{enumitem}
\usepackage{array}
\usepackage{float}
\usepackage{graphicx}
\usepackage[colorlinks]{hyperref}

\graphicspath{{./images/}}

\numberwithin{equation}{section}

\newcommand{\R}{\mathbb{R}}
\newcommand{\Z}{\mathbb{Z}}

\newcommand{\Zreg}{\mathcal{Z}_{\mathrm{reg}}}
\newcommand{\Zhat}{\widehat{\Zreg}}

\newcommand{\Haus}[1]{\mathcal{H}^{#1}}
\newcommand{\Leb}[1]{\mathcal{L}^{#1}}

\newcommand{\tgrad}{\overline{\nabla}}

\renewcommand{\d}{\,\mathrm{d}}

\providecommand{\given}{}
\DeclarePairedDelimiterX{\set}[1]\{\}{\renewcommand{\given}{\,\delimsize\vert\,}#1}

\DeclareMathOperator{\diam}{diam}
\DeclareMathOperator{\supp}{supp}
\DeclareMathOperator{\sgn}{sign}
\DeclareMathOperator{\TV}{TV}
\DeclareMathOperator*{\osc}{osc}
\DeclareMathOperator*{\vosc}{vosc}

\newcommand{\Osc}[1]{\osc_{#1}}
\newcommand{\Vosc}[1]{\vosc_{#1}}

\newtheorem{theorem}{Theorem}[section]
\newtheorem{lemma}[theorem]{Lemma}

\theoremstyle{definition}
\newtheorem{definition}[theorem]{Definition}

\theoremstyle{remark}

\title[Classical Analysis Counterpart of Viterbo's ABP Proof]{A Classical Analysis Counterpart of Viterbo's Symplectic Geometry Proof of ABP in the Plane}

\author[Maienshein]{Daniel Maienshein}
\address{Department of Mathematics, University of Pittsburgh, PA 15260, USA}
\email{dnm48@pitt.edu}

\author[Manfredi]{Juan J.~Manfredi}
\address{Department of Mathematics, University of Pittsburgh, PA 15260, USA}
\email{manfredi@pitt.edu}

\date{May 12, 2026}

\thanks{The research was supported  by the Simons 
Collaboration Gift 962828}

\begin{document}

\begin{abstract}
    We first provide a classical analysis proof of a version of the Alexandroff--Bakelman--Pucci inequality (ABP) for compactly supported $C^2$ functions in dimension $2$, inspired by the symplectic geometry proof method of Viterbo, which avoids convexity or contact sets. We then show how the proof may be modified to remove the compact support hypothesis and recover the usual statement of ABP, which includes a boundary term. We also discuss the possibility (and difficulties) of extending a pure classical analysis proof to dimension $3$ and above.
\end{abstract}

\maketitle

\section{Introduction}

Suppose $\Omega \subseteq \R^n$ and $a^{ij}, b^i, c \colon \Omega \to \R$ are measurable functions. Consider the uniformly elliptic linear operator
\[
    L u \coloneqq a^{ij}(x)\,D_{ij} u + b^i(x)\, D_i u + c(x)\, u,
\]
where $0 < \lambda \leq a^{ij}(x) \leq \Lambda$ for all $x \in \Omega$. The classical Alexandroff--Bakelman--Pucci (ABP) estimate \cite{cabre1995alexandroff, gilbarg1998elliptic} is the following weak maximum principle for strong supersolutions of the equation $Lu = f$.

\begin{theorem}\label{thm:ABP-classical}
    Suppose that $\Omega$ is bounded, $c \leq 0$, $f \in L^n(\Omega)$,  $b_i\in L^\infty(\Omega)$ for each $i$, and $\left(\sum_{i}b_i(x)^2\right)^{1/2} \leq B$ for all $x \in \Omega$. If  $u \in C^0(\overline{\Omega}) \cap W^{2,n}_{\mathrm{loc}}(\Omega)$ satisfies $Lu \geq f$ in $\Omega$, then
    \begin{equation*}\label{eq:ABP-classical-PDE}
        \sup_\Omega u \leq \sup_{\partial\Omega} u + C\, \diam(\Omega)\, \norm{f}_{L^n(\Omega)},
    \end{equation*}
    where $C = C(n, \lambda, B \diam(\Omega))$.
\end{theorem}

\noindent The heart of the proof of Theorem~\ref{thm:ABP-classical} is the following lemma \cite[Lemma 9.2]{gilbarg1998elliptic}, which we shall simply refer to as ABP.

\begin{lemma}\label{lem:ABP-analytic}
    If $u \in C^2(\Omega) \cap C^0(\overline{\Omega})$, then
    \begin{equation}\label{eq:ABP-classical-analytic}
        \bigl(\sup_\Omega u - \sup_{\partial\Omega} u\bigr)^n
        \leq C\, \diam(\Omega)^n \int_{\Omega} \abs{\det D^2 u(x)} \d x.
    \end{equation}
\end{lemma}

\noindent The standard proof of Lemma~\ref{lem:ABP-analytic} has two steps. First, one shows that
\(
    \Leb{n}(D u(\Gamma^+)) \leq \int_{\Gamma^+} \abs{\det D^2 u(x)} \d x
\)
using the area formula, where $\Leb{n}$ is $n$-dimensional Lebesgue measure and $\Gamma^+$ is the upper contact set of $u$\footnote{The set of points in $\Omega$ where $u$ coincides with its concave envelope.}. Then one finishes by estimating
\(
    (\sup_\Omega u - \sup_{\partial\Omega} u)^n \leq C\, \Leb{n}(D u(\Gamma^+)).
\)

In 2000, a new proof emerged when Viterbo established a version of ABP for compactly supported functions using tools from symplectic geometry and the theory of generating functions \cite{viterbo2000metric}. The precise statement comes from Corollary~3.9 of that paper, which we restate here.

\begin{theorem}[{\cite[Corollary 3.9]{viterbo2000metric}}]\label{thm:Viterbo-ABP}
    Let $K$ be a compact subset of $\R^n$. There exists a constant $\rho_n$ depending only on $n$ such that, for all $f \in C^2(\R^n)$ with $\supp f \subseteq K$,
    \[
        \bigl(\Osc{K} f\bigr)^n \leq \rho_n\, \Leb{n}(K) \int_{K} \abs{\det D^2 f(x)} \d x.
    \]
\end{theorem}

\noindent Here, $\Osc{K} f \coloneqq \sup_K f - \inf_K f$. Theorem~\ref{thm:Viterbo-ABP} is proven as a corollary of a more general result about Lagrangian submanifolds of the cotangent bundle $T^*K$ which are Hamiltonian isotopic to the zero section (see \cite[Theorem 3.8]{viterbo2000metric}). Here, we show that Viterbo's proof idea can be reformulated as a classical analysis proof in the function case for $n = 2$, circumventing any direct reference to concepts from symplectic geometry. As a consequence of using Viterbo's proof as a guideline, our proof also avoids reference to convexity and contact sets. Further, we show how the proof may be modified to remove the compact support hypothesis by including a boundary term. (Viterbo believed such a modification to be possible, but he did not explore this question; see \cite{viterbo2000metric}.) Finally, we comment on the $n = 3$ case.

\section{Main Results}

\begin{theorem}[ABP for $n=2$, compactly supported case]\label{thm:ABP-n2-compact}
    Let $K$ be a compact subset of $\R^2$ with nonempty interior. There exists a constant $C_K$ depending on the square of the diameter of $K$ such that, for all $f \in C^2(\R^2)$ with $\supp f \subseteq K$,
    \[
        \bigl(\Osc{K} f\bigr)^2 \leq C_K \int_{K} \abs{\det D^2 f(x)} \d x.
    \]
\end{theorem}

\noindent Our constant $C_K$, which we shall show equals $16\, \diam(K)^2$, depends on the diameter of $K$, not the volume of $K$ as in Viterbo's proof. This implies that while our proof is qualitatively similar to Viterbo's, it is not a one-to-one translation of his proof into the language of classical analysis. At the same time, dependence on the diameter is not surprising, since the constant in the classical ABP estimate does depend on the diameter; see \eqref{eq:ABP-classical-analytic}.

The idea of the proof of Theorem~\ref{thm:ABP-n2-compact} is to use the co-area formula to express the integral of $\abs{\det D^2 f}$ over $K$ as an integral of the total variation of $f_{x_1}$ on the level sets of $f_{x_2}$. The problem is then reduced to many one-dimensional problems, where we need a one-dimensional version of ABP. That is the content of Lemma~\ref{lem:ABP-1D-loop}. We begin with a definition.

\begin{definition}[Vertical oscillation]\label{def:vosc}
    For $u \in C(\R^2)$ and a nonempty bounded subset $A \subseteq \R^2$, define the \emph{vertical oscillation} of $u$ on $A$ to be
    \[
        \Vosc{A}(u) \coloneqq \sup\set*{u(x) - u(y) \given x, y \in A,\ p_1(x) = p_1(y)},
    \]
    where for $i = 1, 2$, we let $p_i \colon \R^2 \to \R$ be the projection $(x_1, x_2) \mapsto x_i$.
\end{definition}

\begin{lemma}[Vertical ABP on a smooth curve]\label{lem:ABP-1D-loop}
    Suppose that $M$ is a smooth, compact, and connected one-dimensional submanifold of $\R^2$ (with or without boundary). Parametrize $M$ by a curve $\gamma(s)$ with $\abs{\gamma'(s)} = 1$ for all $s \in [0,L]$, where $L = \Haus{1}(M)$. If $u \in C^2(\R^2)$ and $M \subseteq u_{x_2}^{-1}(0)$, then
    \[
        \Vosc{M}(u) \leq \Haus{1}(p_1(M)) \int_{0}^{L} \abs{\dv{s} u_{x_1}(\gamma(s))} \d s.
    \]
\end{lemma}

\noindent Notably, the constant depends only on the projected length of $M$.

\begin{proof}[Proof of Lemma~\ref{lem:ABP-1D-loop}]
    Take any two points $x, y \in M$ with $p_1(x) = p_1(y)$. Choose $s_1, s_2 \in [0,L]$ such that $\gamma(s_1) = y$ and $\gamma(s_2) = x$. Note that $p_1 \gamma(s_1) = p_1 \gamma(s_2) = p_1(x)$. Using the fact that $u_{x_2}(\gamma(s)) = 0$ for all $s$ and integrating by parts, we find
    \begin{align*}
        u(x) - u(y)
        &= \int_{s_1}^{s_2} u_{x_1}(\gamma(s))\,(p_1 \gamma)'(s) \d s \\
        &= \bigl[u_{x_1}(\gamma(s_2)) - u_{x_1}(\gamma(s_1))\bigr]\, p_1(x)
            - \int_{s_1}^{s_2} \dv{s} u_{x_1}(\gamma(s))\, p_1 \gamma(s) \d s \\
        &= \int_{s_1}^{s_2} \dv{s} u_{x_1}(\gamma(s))\, p_1(x) \d s
            - \int_{s_1}^{s_2} \dv{s} u_{x_1}(\gamma(s))\, p_1 \gamma(s) \d s \\
        &= \int_{s_1}^{s_2} \dv{s} u_{x_1}(\gamma(s))\, p_1(x - \gamma(s)) \d s \\
        &\leq \Haus{1}(p_1(M)) \int_{0}^{L} \abs{\dv{s} u_{x_1}(\gamma(s))} \d s.
    \end{align*}
    We conclude by taking the supremum over $x$ and $y$.
\end{proof}

\subsection{Definitions, notation, and a topological lemma}\label{sec:defn-not-top}

The next lemma is a topological existence result which will be used in the final step of the proof of Theorem~\ref{thm:ABP-n2-compact}. In order to state the lemma, we need to make several definitions and define some notation.

Fix a compact subset $K$ of $\R^2$ with nonempty interior, and suppose that $\Sigma = \bigcup_{i=1}^{m} \Omega_i \subseteq K$ is a smooth, compact (embedded) $1$-dimensional submanifold of $\R^2$ which is a disjoint union of simple closed curves $\Omega_i$. Define the \emph{crossing number} $X_\Sigma \colon (K \setminus \Sigma) \times (K \setminus \Sigma) \to \Z_{\geq 0}$ by
\[
    X_\Sigma(x, y) \coloneqq \inf\set*{\#\bigl(\gamma([0,1]) \cap \Sigma\bigr) \given \gamma \in C^\infty([0,1]; K),\ \gamma(0) = x,\ \gamma(1) = y}.
\]
Next, for a fixed $x_0 \in K \setminus \Sigma$, define a \emph{binary coloring} of the relatively open set $K \setminus \Sigma$ with respect to $x_0$ to be one of the two functions
\[
    c_{0, x_0}(y) \coloneqq (-1)^{X_\Sigma(x_0, y)}, \qquad c_{1, x_0} \coloneqq -c_{0, x_0},
\]
for all $y \in K \setminus \Sigma$. Hence points $y \in K \setminus \Sigma$ with even crossing number with respect to $x_0$ satisfy $c_{0, x_0}(y) = +1$ and $c_{1, x_0}(y) = -1$, while points with odd crossing number satisfy $c_{0, x_0}(y) = -1$ and $c_{1, x_0}(y) = +1$.

Fix a point $x_0 \in K \setminus \Sigma$ and a coloring $c \in \set{c_{0, x_0}, c_{1, x_0}}$. Each $\Omega_i$ has a natural orientation (choice of outer unit normal vector) defined by the normal vector $N$ which points into the unbounded component of $\R^2 \setminus \Omega_i$. We call $\Omega_i$ \emph{$c$-positive} if for all $\delta$ sufficiently close to $0$,
\[
    c\left(\set*{y + \delta N(y) \given y \in \Omega_i} \right)\subseteq \set{\sgn(\delta)}.
\]
Similarly, we call $\Omega_i$ \emph{$c$-negative} if for all $\delta$ sufficiently close to $0$,
\[
    c\left(\set*{y + \delta N(y) \given y \in \Omega_i}\right) \subseteq \set{-\sgn(\delta)}.
\]

A bit more notation shall prove helpful. For $x, y \in K$, we write $x \prec y$ (respectively, $x \succ y$) if $p_2(x) \leq p_2(y)$ (respectively, $p_2(x) \geq p_2(y)$) \emph{and} $p_1(x) = p_1(y)$. Also, if $x \in \Sigma$, let $\Omega_{i(x)}$ denote the connected component of $\Sigma$ containing $x$. We may now define $G \colon \Sigma \to \Sigma$ by
\[
    G(x) \coloneqq
    \begin{cases}
        \bigl(p_1(x),\ \sup\set{p_2(y) \given x \prec y \in \Omega_{i(x)}}\bigr) & \text{if } \Omega_{i(x)} \text{ is $c$-positive,} \\[4pt]
        \bigl(p_1(x),\ \inf\set{p_2(y) \given x \succ y \in \Omega_{i(x)}}\bigr) & \text{if } \Omega_{i(x)} \text{ is $c$-negative.}
    \end{cases}
\]
Intuitively, $G(x)$ is the highest point above $x$ on a $c$-positive loop, and the lowest point below $x$ on a $c$-negative loop. Hence $G$ depends on the choice of $c$.

Further, for $x, y \in K$, let $\ell(x, y)(t) \coloneqq (1-t)x + t y$ for $t \in [0,1]$, denote its image by $\ell(x, y) \coloneqq \ell(x, y)([0,1])$, and call the interior of the image $\ell(x, y)^\circ \coloneqq \ell(x, y)((0,1))$. Moreover, for $x, y \in \Omega_i$, denote by $\gamma(x, y)$ one of the two unit-speed parametrizations of the path inside $\Omega_i$ with initial point $x$ and final point $y$. (It does not matter which of the two paths we choose.) A path $\alpha$ in $K$ shall be called \emph{$(\Sigma, c)$-admissible} if any one of the following statements is true:
\begin{enumerate}[label=(\roman*)]
    \item The path $\alpha$ equals $\ell(x, y)$ for some $y \prec x$ such that $\ell(x, y)^\circ \subseteq c^{-1}(+1)$.
    \item The path $\alpha$ equals $\ell(x, y)$ for some $y \succ x$ such that $\ell(x, y)^\circ \subseteq c^{-1}(-1)$.
    \item The path $\alpha$ equals $\gamma(G(x), x)$ for some $i \in \set{1, \ldots, m}$ and $x \in \Omega_i$.
\end{enumerate}
In other words, admissible paths are either vertical segments that move down in regions with color $+1$, vertical segments that move up in regions with color $-1$, or sub-arcs of a loop $\Omega_i$ maximal with respect to endpoints that have the same $p_1$-projection as $x$. Finally, recall that the concatenation of two paths $\alpha$ and $\beta$ is denoted by $\alpha * \beta$, where $\alpha$ is traversed before $\beta$.

\begin{lemma}\label{lem:path-existence}
    Suppose that $\Sigma = \bigcup_{i=1}^{m} \Omega_i \subseteq K$ is a disjoint union of smooth, simple closed curves $\Omega_i$ in the interior of some compact set $K \subset \R^2$. For all $x^* \in K \setminus \Sigma$ and for each binary coloring $c \in \set{c_{0, x^*}, c_{1, x^*}}$ of $K \setminus \Sigma$, there exists a piecewise-smooth path $\Gamma = \alpha_1 * \alpha_2 * \cdots * \alpha_b$ for some finite integer $b$ and $(\Sigma, c)$-admissible paths $\alpha_i$, with the following properties:
    \begin{enumerate}[label=(\arabic*)]
        \item The initial point of $\Gamma$ is in $\partial K$ and the final point of $\Gamma$ is $x^*$.
        \item If $\Gamma$ exits a given connected component of $\Sigma$, then $\Gamma$ never returns to that same component. (Formally, $\Gamma(s_1) \in \Omega_i$ and $\Gamma(s_1 + \epsilon) \notin \Omega_i$ for small $\epsilon > 0$ implies that $\Gamma(s) \notin \Omega_i$ for all $s > s_1$.)
    \end{enumerate}
\end{lemma}

For a proof of Lemma~\ref{lem:path-existence}, as well as pictures of these paths for the different colorings, see the Appendix. (While Lemma~\ref{lem:path-existence} does not seem to be a deep result of plane topology or geometry, to the best of our knowledge, the result is absent from the literature.)

\subsection{Proof of ABP for \texorpdfstring{$n=2$}{n=2}, compactly supported case}

\begin{proof}[Proof of Theorem~\ref{thm:ABP-n2-compact}]

\textbf{Step 0.}\quad
By a standard approximation argument, we may assume that $f \in C^\infty(\R^2)$.

\medskip\noindent\textbf{Step 1.}\quad
We prove that
\begin{equation}\label{ineq:Step1}
    \boxed{\int_{K} \abs{\det D^2 f(x)} \d x = \int_{-\infty}^{\infty} \TV_z(f_{x_1}) \d z,}
\end{equation}
where $\Sigma_z \coloneqq f_{x_2}^{-1}(z)$ and $\TV_z(f_{x_1})$ is the total variation of $f_{x_1}$ on $\Sigma_z$. Since $n = 2$, we have $\abs{\det D^2 f} = \abs{\nabla f_{x_1} \cdot X_{f_{x_2}}}$, where $X_{f_{x_2}} \coloneqq (-f_{x_2 x_2},\, f_{x_1 x_2})$. Note that $\abs{X_{f_{x_2}}} = \abs{\nabla f_{x_2}}$, so
\[
    \frac{\abs{\det D^2 f}}{\abs{\nabla f_{x_2}}}
    = \abs{\nabla f_{x_1} \cdot X_{f_{x_2}}/\abs*{X_{f_{x_2}}}}.
\]
Therefore, by the co-area formula \cite[Theorem 3.2.22]{federer1969},
\begin{equation}\label{eq:coarea}
    \int_{K} \abs{\det D^2 f(x)} \d x
    = \int_{-\infty}^{\infty} \int_{\Sigma_z} \abs{\nabla f_{x_1} \cdot X_{f_{x_2}}/\abs*{X_{f_{x_2}}}} \d \Haus{1} \d z.
\end{equation}
By Step 0 we may assume that $f_{x_2}$ is smooth, so Sard's theorem \cite[Theorem 6.10]{lee2003smooth} implies that $f_{x_2}^{-1}(z)$ is a regular level set for a.e.\ $z \in \R$. Denote by $\Zreg$ the set of regular values of $f_{x_2}$. By the regular value theorem \cite[Theorem 1.38]{warner1983}, $f_{x_2}^{-1}(z)$ is a smooth (embedded) codimension-$1$ submanifold of $\R^2$ for all $z \in \Zreg$. Since $f_{x_1}$ is smooth,
\[
    \TV_z(f_{x_1}) = \int_{\Sigma_z} \abs{\nabla_z f_{x_1}} \d \Haus{1},
\]
where $\nabla_z$ is the tangential derivative on $\Sigma_z$. We claim that
\begin{equation}\label{eq:TV}
    \TV_z(f_{x_1}) = \int_{\Sigma_z} \abs{\nabla f_{x_1} \cdot X_{f_{x_2}}/\abs*{X_{f_{x_2}}}} \d \Haus{1}.
\end{equation}
Since $\nabla f_{x_2}$ is normal to $\Sigma_z$ for regular $z$ and $\nabla f_{x_2} \cdot X_{f_{x_2}} = 0$, the unit vector field $X_{f_{x_2}} / \abs*{X_{f_{x_2}}}$ is tangent to $\Sigma_z$, so
\[
    \nabla_z f_{x_1} = \nabla f_{x_1} \cdot \frac{X_{f_{x_2}}}{\abs*{X_{f_{x_2}}}}.
\]
This establishes \eqref{eq:TV}. Inserting \eqref{eq:TV} into \eqref{eq:coarea} yields \eqref{ineq:Step1}.

\medskip\noindent\textbf{Step 2.}\quad
Let $z^* \coloneqq \sup_K f_{x_2}$, $z_* \coloneqq \inf_K f_{x_2}$, and $\Zhat \coloneqq \Zreg \cap [z_*, z^*] \setminus \set{0}$. For any $z \in \Zhat$, the compact support hypothesis implies that
\[
    \Sigma_z = \bigsqcup_{i=1}^{m_z} \Omega_{z, i},
\]
where $m_z$ is a positive integer and each $\Omega_{z, i} \subset K$ is a smooth, simple closed curve. For each $z \in \Zhat$, define\footnote{The function $S_z$ is known as a generating function for the Lagrangian submanifold $\set{(p_1(x),\, f_{x_1}(x)) \given x \in \Sigma_z} \subseteq T^*K$; see \cite{viterbo1992symplectic, viterbo2000metric} for definitions. For the purposes of our analysis, we only need to view $S_z$ as an appropriate modification of $f$ which satisfies the hypotheses of Lemma~\ref{lem:ABP-1D-loop}.}
\[
    S_z(x) \coloneqq f(x) - z\, p_2(x), \qquad x \in K.
\]
Note that $\partial S_z / \partial x_2 = 0$ on $\Sigma_z$ and $\partial S_z / \partial x_1 = f_{x_1}$ on $K$. Define $\phi \colon \R \to [0, \infty)$ by
\[
    \phi(z) \coloneqq
    \begin{cases}
        \displaystyle \sum_{i=1}^{m_z} \Vosc{\Omega_{z, i}} S_z & \text{if } z \in \Zhat, \\[4pt]
        0 & \text{otherwise.}
    \end{cases}
\]
We claim that
\begin{equation}\label{ineq:Step2}
    \boxed{\phi(z) \leq \diam(K)\, \TV_z(f_{x_1}) \quad \text{for all } z \in \R.}
\end{equation}
The claim is trivial for $z \notin \Zhat$, so suppose $z \in \Zhat$. The function $S_z$ and the loops $\Omega_{z, i}$ satisfy the hypotheses of Lemma~\ref{lem:ABP-1D-loop}. Hence for $1 \leq i \leq m_z$,
\[
    \Vosc{\Omega_{z, i}} S_z
    \leq \Haus{1}(p_1(\Omega_{z, i})) \int_{\Omega_{z, i}} \abs*{\dv{s} f_{x_1}} \d \Haus{1}(s).
\]
Noting that $\Haus{1}(p_1(\Omega_{z, i})) \leq \diam(K)$, summing in $i$ and using the disjointness of the $\Omega_{z, i}$, the claim follows.

\medskip\noindent\textbf{Step 3.}\quad
We claim that
\begin{equation}\label{ineq:osc-phi}
    \boxed{\bigl(\Osc{K} f\bigr)^2 \leq 16\, \diam(K) \int_{-\infty}^{\infty} \phi(z) \d z.}
\end{equation}
The proof is adapted from \cite[Proposition 2.2]{viterbo2000metric}. Temporarily fix a real number $a > \int_{-\infty}^{\infty} \phi(z) \d z$, and fix $\lambda > 0$ (to be determined). Define
\[
    E_\lambda \coloneqq \set{z \in \Zhat \given \phi(z) \geq \lambda}.
\]
Choose $x^*, x_* \in K$ such that $f(x^*) = \sup_K f$ and $f(x_*) = \inf_K f$. Without loss of generality, at least one of $x^*$ and $x_*$ is an interior critical point of $f$ (otherwise, since $f$ vanishes on $\partial K$, there is nothing to prove). We may assume $x^*$ is interior to $K$ for concreteness; note that $x^* \in \Sigma_0$. By Markov's inequality,
\[
    \abs{E_\lambda}
    \leq \frac{1}{\lambda} \int_{\set{z \in \Zhat \given \phi(z) \geq \lambda}} \phi(z) \d z
    < \frac{a}{\lambda}.
\]

We first analyze the case $\Zhat \setminus E_\lambda = \varnothing$. Here $E_\lambda = \Zhat$, so $z^* - z_* = \abs{E_\lambda} < a/\lambda$. Since $f_{x_2}(x^*) = 0$, we have $0 \in [z_*, z^*]$, hence $z^* < a/\lambda$. Setting $q \coloneqq (p_1(x^*),\, \sup\set{p_2(y) \given y \in \partial K \text{ and } y \prec x^*}) \in \partial K$, we have $\ell(q, x^*) \subseteq K$. (Note carefully that convexity of $K$ was not assumed and is not needed here, since $x^*$ is an interior point of $K$.) Using the fact that $f$ vanishes on $\partial K$,
\begin{multline*}
    \sup_K f = \sup_K f - f(q) = \int_{p_2(q)}^{p_2(x^*)} f_{x_2}(p_1(x^*), s) \d s \\
    \leq z^*\, p_2(x^* - q) \leq \frac{a}{\lambda}\, \diam(K).
\end{multline*}
By the same reasoning, we get $-\inf_K f \leq (a/\lambda)\, \diam(K)$. Hence we obtain $\Osc{K} f \leq 2(a/\lambda)\, \diam(K)$. As we shall see, this estimate is sharper than what we get in the case $\Zhat \setminus E_\lambda \neq \varnothing$.

Therefore, without loss of generality, assume $\Zhat \setminus E_\lambda \neq \varnothing$. Since $\abs{E_\lambda} < a/\lambda$, there exists $z_0 \in \Zhat \setminus E_\lambda$ such that $\abs{z_0} \leq a/\lambda$. (Otherwise, there would exist an open interval $J$ containing $0$ such that $\abs{J \cap E_\lambda} \geq a/\lambda$, which is a contradiction.) We estimate $\Osc{K} f$ by estimating $\sup_K f$ and $\inf_K f$ separately, but we provide details for the supremum only, since the argument for the infimum is identical. We prove that $\sup_K f \leq (a/\lambda)\, \diam(K) + \lambda$.

Choose a binary coloring $c \in \set{c_{0, x^*}, c_{1, x^*}}$ (see Section~\ref{sec:defn-not-top}) for the relatively open set $K \setminus \Sigma_{z_0}$ according to the sign of $z_0$:
\[
    c \coloneqq
    \begin{cases}
        c_{0, x^*} & \text{if } z_0 < 0, \\
        c_{1, x^*} & \text{if } z_0 > 0.
    \end{cases}
\]
With this choice, we can control the change in $f$ over $(\Sigma_{z_0}, c)$-admissible vertical segments by vertical displacements; this is made precise in the following lemma.

\begin{lemma}\label{lem:change-over-vertical-admissible-segments}
    For any $(\Sigma_{z_0}, c)$-admissible line segment $\ell(x, y)$,
    \[
        f(y) - f(x) \leq z_0\, p_2(y - x).
    \]
\end{lemma}

\begin{proof}[Proof of Lemma~\ref{lem:change-over-vertical-admissible-segments}]
    First suppose that $z_0 > 0$, so $c = c_{1, x^*}$. Since $x^*$ is a critical point of $f$, $f_{x_2}(x^*) = 0 < z_0$. Therefore, by continuity, $f_{x_2} < z_0$ inside the connected component of $K \setminus \Sigma_{z_0}$ containing $x^*$. The coloring $c$ is constant on that component, so by definition of $c_{1, x^*}$ that component has color $c_{1, x^*}(x^*) = -1$. Moreover, the outer unit normal vector of $\Sigma_{z_0}$ is a scalar multiple of $\nabla f_{x_2}$, which does not vanish since $z_0 \in \Zreg$. Hence the crossing number always jumps by $\pm 1$ across $\Sigma_{z_0}$, and so
    \[
        \begin{aligned}
            c^{-1}(-1) &= \set{x \in K \setminus \Sigma_{z_0} \given f_{x_2}(x) < z_0}, \\
            c^{-1}(+1) &= \set{x \in K \setminus \Sigma_{z_0} \given f_{x_2}(x) > z_0}.
        \end{aligned}
    \]
    Similarly, the reader can verify that if $z_0 < 0$ (so $c = c_{0, x^*}$), the same characterizations hold.

    Since $\ell(x, y)$ is $(\Sigma_{z_0}, c)$-admissible, either $y \prec x$ or $y \succ x$. Suppose $y \succ x$. Then $p_2(y - x) \geq 0$ and $\ell(x, y)^\circ \subseteq c^{-1}(-1) = \set{f_{x_2} < z_0}$. Hence
    \begin{align*}
        f(y) - f(x)
        &= \int_{0}^{1} \dv{s} f((1-s)x + sy) \d s \\
        &= \abs{p_2(y - x)} \int_{0}^{1} f_{x_2}((1-s)x + sy) \d s \\
        &\leq \abs{p_2(y - x)} \int_{0}^{1} z_0 \d s \\
        &= z_0\, p_2(y - x).
    \end{align*}
    Similarly, $f(y) - f(x) \leq z_0\, p_2(y - x)$ holds whenever $y \prec x$.
\end{proof}

Now we can finish Step 3. Let $\Gamma$ be the path from some point $q \in \partial K$ to $x^*$ provided by Lemma~\ref{lem:path-existence} corresponding to the coloring $c$. Let $x_0 = q,\, x_1, \ldots, x_{k-1},\, x_k = x^*$ be the endpoints of each admissible segment of $\Gamma$. Let us write
\[
    f(x^*) - f(q) = \sum_{j=1}^{k} \bigl[f(x_j) - f(x_{j-1})\bigr].
\]
By Lemma~\ref{lem:change-over-vertical-admissible-segments}, if $x_j$ and $x_{j-1}$ are endpoints of some admissible $\ell(x, y)$, then $f(x_j) - f(x_{j-1}) \leq z_0\, p_2(x_j - x_{j-1})$. The only other possibility is that $x_{j-1} = G(x_j)$ for some $x_j \in \Omega_i$ and some $i$, in which case
\begin{align*}
    f(x_j) - f(x_{j-1})
    &= S_{z_0}(x_j) - S_{z_0}(G(x_j)) + z_0\, p_2(x_j - x_{j-1}) \\
    &\leq \Vosc{\Omega_{z_0, i}} S_{z_0} + z_0\, p_2(x_j - x_{j-1}).
\end{align*}
By Lemma~\ref{lem:path-existence}, Property~(2), the path $\Gamma$ visits any given connected component at most once; hence the sum of the $\Vosc{\Omega_{z_0, i}} S_{z_0}$ is at most $\phi(z_0)$. It follows that
\[
    f(x^*) - f(q) \leq \phi(z_0) + z_0 \sum_{j=1}^{k} p_2(x_j - x_{j-1})
    = \phi(z_0) + z_0\, p_2(x^* - q).
\]
Hence the total change in $f$ over $\Gamma$ is at most $\phi(z_0) + \abs{z_0}\, \diam(K)$. Since $z_0 \notin E_\lambda$, $\phi(z_0) \leq \lambda$, and since $\abs{z_0} \leq a/\lambda$, we obtain
\begin{equation*}
    \sup_K f = f(x^*) - f(q) = \int_{0}^{\Haus{1}(\Gamma)} \dv{s} f(\Gamma(s)) \d \Haus{1} \leq \frac{a}{\lambda}\, \diam(K) + \lambda.
\end{equation*}
The same argument shows $-\inf_K f \leq (a/\lambda)\, \diam(K) + \lambda$, so
\[
    \Osc{K} f \leq 2(a/\lambda)\, \diam(K) + 2\lambda.
\]
The minimum of the function $\lambda \mapsto (a/\lambda)\, \diam(K) + \lambda$ occurs at $\lambda = \sqrt{a\, \diam(K)}$, so choosing this $\lambda$ we find
\[
    \Osc{K} f \leq 4\sqrt{a\, \diam(K)}.
\]
Squaring both sides and letting $a$ tend to $\int_{-\infty}^{\infty} \phi(z) \d z$, we obtain \eqref{ineq:osc-phi}. Inserting \eqref{ineq:Step2} into \eqref{ineq:osc-phi}, and finally using \eqref{ineq:Step1}, we conclude that
\begin{align*}
    \bigl(\Osc{K} f\bigr)^2
    &\leq 16\, \diam(K) \int_{-\infty}^{\infty} \diam(K)\, \TV_z(f_{x_1}) \d z.
\end{align*}

\noindent Theorem \ref{thm:ABP-n2-compact} now follows from Steps 0 through 3 with $C_K = 16\, \diam(K)^2$.  \end{proof}

\subsection{ABP for \texorpdfstring{$n=2$}{n=2} in the general case}

Without the compact support assumption, we need to incorporate a boundary term into the ABP estimate. As Viterbo suspected, the same proof strategy works, leading to:

\begin{theorem}[ABP for $n=2$, general case]\label{thm:ABP-n2-boundary}
    Let $K$ be a compact subset of $\R^2$ with nonempty interior. There exists a constant $C_K$ depending on the square of the diameter of $K$ such that, for all $f \in C^2(\R^2)$,
    \[
        \bigl(\Osc{K} f - \Osc{\partial K} f\bigr)^2 \leq C_K \int_{K} \abs{\det D^2 f(x)} \d x.
    \]
\end{theorem}

Before proving Theorem~\ref{thm:ABP-n2-boundary}, we need to do a bit more preparatory work. We start with a definition. If $J$ is a smooth embedded curve with arc length parametrization $\gamma \colon [0, L] \to \R^2$, for each $x \in J$ there exists a unique $s_x \in [0, L]$ such that $\gamma(s_x) = x$. Define $\hat{G} \colon J \to J$ by
\[
    \hat{G}(x) \coloneqq \gamma(\sigma_x), 
\]

\noindent where
\[ \sigma_x \coloneqq \min\set{s_x, \,\inf\set{s >0 \given p_1(\gamma(s)) = p_1(x)}}.
\] 

\noindent Note that $\hat{G}$ depends on the orientation of $J$ induced by the direction of parametrization. We now prove the following corollary of Lemma~\ref{lem:ABP-1D-loop}.

\begin{lemma}[ABP on a smooth curve with non-vanishing $x_1$-derivative condition]\label{lem:ABP-1D-curve-nonvanishing}
    Suppose that $J$ is a smooth, compact, and connected one-dimensional submanifold of $\R^2$ with boundary. Parametrize $J$ by a curve $\gamma(s)$ with $\abs{\gamma'(s)} = 1$ for all $s \in [0, L]$, where $L = \Haus{1}(J)$. Suppose that $u \in C^2(\R^2)$ and $J \subseteq u_{x_2}^{-1}(0) \cap u_{x_1}^{-1}(\R \setminus \set{0})$. Let $\mathcal{P}$ be the set of partitions of $[0, L]$ of the form
    \[
        s_1 \geq t_1 \geq s_2 \geq t_2 \geq \cdots \geq s_k \geq t_k,
    \]
    where $t_i$ are chosen such that $\hat{G}(\gamma(s_i)) = \gamma(t_i)$. Then
    \[
        \sup_{\mathcal{P}} \sum_{i} \bigl[u(\gamma(s_i)) - u(\gamma(t_i))\bigr]
        \leq \Haus{1}(p_1(J)) \int_{0}^{L} \abs{\dv{s} u_{x_1}(\gamma(s))} \d s.
    \]
\end{lemma}

\begin{proof}[Proof of Lemma~\ref{lem:ABP-1D-curve-nonvanishing}]
    For each $i$, by definition of $\hat{G}$ we see that $\gamma(t_i)$ lies vertically above or below $\gamma(s_i)$. Similarly to Lemma~\ref{lem:ABP-1D-loop}, we can use integration by parts to show
    \[
        u(\gamma(s_i)) - u(\gamma(t_i))
        \leq \Haus{1}(p_1(\gamma[t_i, s_i])) \int_{t_i}^{s_i} \abs{\dv{s} u_{x_1}(\gamma(s))} \d s.
    \]
    The intervals $(t_i, s_i)$ are disjoint by nature of the partition $\mathcal{P}$, so the result follows by summing in $i$ and taking the supremum over all partitions.
\end{proof}

We shall also make use of the following lemma.

\begin{lemma}[ABP on a smooth curve with vanishing $x_1$-derivative condition]\label{lem:ABP-1D-curve-vanishing}
    Suppose that $J$ is a smooth, compact, and connected one-dimensional submanifold of $\R^2$ with boundary. Parametrize $J$ by a curve $\gamma(s)$ with $\abs{\gamma'(s)} = 1$ for all $s \in [0, L]$, where $L = \Haus{1}(J)$. Suppose that $u \in C^2(\R^2)$, $J \subseteq u_{x_2}^{-1}(0)$, and $J \cap u_{x_1}^{-1}(0) \neq \varnothing$. Then there exists $q \in \set{\gamma(0), \gamma(L)}$ such that
    \[
        \sup_J u - u(q) \leq \Haus{1}(p_1(J)) \int_{0}^{L} \abs{\dv{s} u_{x_1}(\gamma(s))} \d s.
    \]
\end{lemma}

\begin{proof}[Proof of Lemma~\ref{lem:ABP-1D-curve-vanishing}]
    Take any $x \in J$ and select $s_1 \in [0, L]$ such that $\gamma(s_1) = x$. By hypothesis, there exists $s^* \in [0, L]$ such that $u_{x_1}(\gamma(s^*)) = 0$. Set
    \[
        s_0 \coloneqq
        \begin{cases}
            0 & \text{if } s^* \leq s_1, \\
            L & \text{if } s^* > s_1,
        \end{cases}
    \]
    and $q \coloneqq \gamma(s_0)$. Note that $s^*$ lies between $s_0$ and $s_1$, so the intervals $(\min\set{s_0, s^*},\, \max\set{s_0, s^*})$ and $(\min\set{s_1, s^*},\, \max\set{s_1, s^*})$ are disjoint. Using $u_{x_2}(\gamma(s)) = 0$ for all $s$ and $u_{x_1}(\gamma(s^*)) = 0$, integrating by parts:
    \begin{align*}
        u(x) - u(q)
        &= \int_{s_0}^{s^*} u_{x_1}(\gamma(s))\, (p_1 \gamma)'(s) \d s \\
        &\qquad + \int_{s^*}^{s_1} u_{x_1}(\gamma(s))\, (p_1 \gamma)'(s) \d s \\
        &= \bigl[-u_{x_1}(\gamma(s_0)) + u_{x_1}(\gamma(s^*))\bigr]\, p_1 \gamma(s_0) \\
        &\qquad - \int_{s_0}^{s^*} \dv{s} u_{x_1}(\gamma(s))\, p_1 \gamma(s) \d s \\
        &\qquad + \bigl[u_{x_1}(\gamma(s_1)) - u_{x_1}(\gamma(s^*))\bigr]\, p_1 \gamma(s_1) \\
        &\qquad - \int_{s^*}^{s_1} \dv{s} u_{x_1}(\gamma(s))\, p_1 \gamma(s) \d s \\
        &= \int_{s_0}^{s^*} \dv{s} u_{x_1}(\gamma(s))\, p_1[\gamma(s_0) - \gamma(s)] \d s \\
        &\qquad - \int_{s^*}^{s_1} \dv{s} u_{x_1}(\gamma(s))\, p_1[\gamma(s_1) - \gamma(s)] \d s \\
        &\leq \Haus{1}(p_1(J)) \int_{0}^{L} \abs{\dv{s} u_{x_1}(\gamma(s))} \d s.
    \end{align*}
    We conclude by taking the supremum over $x \in J$.
\end{proof}

Let us proceed with a definition. Let $K$ be a compact set in $\R^2$, and let $J$ be a smooth one-dimensional submanifold of $\R^2$ with boundary such that $J \subseteq K$ and $\partial J \subseteq \partial K$. Suppose that $\gamma \colon [0, L] \to K$ is a smooth arc-length parametrization of $J$. We call $s \in (0, L)$ a \emph{folding point} of $\gamma$ if:
\begin{itemize}
    \item $s$ is a critical point of $p_1(\gamma)$,
    \item $s^+ \coloneqq \inf_{\sigma > s} \set{p_1(\gamma'(\sigma)) \neq 0} < L$,
    \item $s^- \coloneqq \sup_{\sigma < s} \set{p_1(\gamma'(\sigma)) \neq 0} > 0$, and
    \item $p_1(\gamma'(s^+ + \epsilon))\, p_1(\gamma'(s^- - \epsilon)) < 0$ for sufficiently small $\epsilon > 0$.
\end{itemize}
Note that generically (for isolated critical points of $p_1(\gamma)$) we have $s^+ = s^- = s$. In case we want to talk about the image of the curve $\gamma$, we also call $x \in J$ a \emph{folding point} if there exists a smooth arc-length parametrization of $J$ such that $x = \gamma(s^*)$ for some folding point $s^*$. For example, the point $(0, 0) \in \R^2$ is a folding point on the curve $x = y^2$, but it is only a critical point that is not a folding point on the curve $x = y^3$.

For $x \in J$, let $s_x \in [0,L]$ be such that  $\gamma(s_x) = x$. Now define a mapping $\hat{H} \colon J \to \set{\text{folding points}} \cup (\partial K)$ by
\[
    \hat{H}(x) \coloneqq
        \gamma(\tau_x),
\]
where 
\[\tau_x = \max\set{0, \sup\set{s < s_x \given s \text{ is a folding point of } \gamma}}.
\]
\noindent Note that $\hat{H}$ depends on the orientation of $J$ induced by the direction of parametrization. Finally, for $x, y \in J$, let $\gamma(x, y)$ be the arc of $J$ such that $\gamma(x, y)(s_x) = x$ and $\gamma(x, y)(s_y) = y$, where $s_x \leq s_y$.

Next we need a topological existence result akin to Lemma~\ref{lem:path-existence} which accounts for curves which meet the boundary of $K$. Suppose that
\[
    \Sigma = \bigsqcup_{i=1}^{m} \Omega_i \;\sqcup\; \bigsqcup_{i=1}^{d} J_i,
\]
where $m$ and $d$ are positive integers, $\Omega_i$ are simple closed curves, and $J_i$ are simple curves intersecting $\partial K$. Fix a coloring $c$ of $K \setminus \Sigma$ and a continuous function $F \colon K \to \R$. Let $G \colon \bigsqcup_{i=1}^{m} \Omega_i \to \bigsqcup_{i=1}^{m} \Omega_i$ be the same function as in the compactly supported case. A path $\alpha$ in $K$ shall be called \emph{$(\Sigma, c, F)$-admissible} if any one of the following statements is true:
\begin{enumerate}[label=(\roman*)]
    \item The path $\alpha$ equals $\ell(x, y)$ for some $y \prec x$ such that $\ell(x, y)^\circ \subseteq c^{-1}(+1)$.
    \item The path $\alpha$ equals $\ell(x, y)$ for some $y \succ x$ such that $\ell(x, y)^\circ \subseteq c^{-1}(-1)$.
    \item The path $\alpha$ equals $\gamma(G(x), x)$ for some $x \in \Omega_i$ and $i \in \set{1, \ldots, m}$.
    \item The path $\alpha$ equals $\gamma(\hat{G}(x), x)$ for some $x \in J_i$ and $i \in \set{1, \ldots, d}$.
    \item The path $\alpha$ equals $\gamma(q, x)$ for some $x \in J_i$, $q \in J_i \cap (\partial K)$, and $i \in \set{1, \ldots, d}$, such that $F(y) = 0$ for some $y \in J_i$ which lies between $x$ and $q$ (meaning $s_x \leq s_y \leq s_q$ or $s_q \leq s_y \leq s_x$).
    \item The path $\alpha$ equals $\gamma(\hat{H}(x), x)$ for some $x \in J_i$ and $i \in \set{1, \ldots, d}$ such that $F$ does not vanish on $J_i$, and $\gamma(\hat{H}(x), x)$ is parametrized so that
    \[
        F(\gamma(\hat{H}(x), x)(s))\, p_1\bigl(\gamma(\hat{H}(x), x)'(s)\bigr) \leq 0.
    \]
\end{enumerate}

It is worth explaining the sixth type of admissible path. Proceeding in reverse, such a curve goes from $x$ to the nearest folding point (or the boundary of $K$ if there is none). By continuity, $F$ must have a definite sign on $J_i$, and since $\gamma(\hat{H}(x), x)$ lies between turning points, $p_1(\gamma(\hat{H}(x), x)'(s))$ also has a definite sign, so we can choose the parametrization which satisfies the monotonicity condition $F(\gamma)\, p_1(\gamma') \leq 0$.

\begin{lemma}\label{lem:path-existence-boundary}
    Suppose that
    \[
        \Sigma = \bigsqcup_{i=1}^{m} \Omega_i \;\sqcup\; \bigsqcup_{i=1}^{d} J_i,
    \]
    where $m$ and $d$ are positive integers, $\Omega_i$ are smooth, simple closed curves in the interior of a compact set $K \subset \R^2$, and $J_i$ are smooth curves intersecting $\partial K$. For all $x^* \in K \setminus \Sigma$, for each binary coloring $c \in \set{c_{0, x^*}, c_{1, x^*}}$ of $K \setminus \Sigma$, and for every continuous function $F \colon K \to \R$, there exists $q \in \partial K$ and a piecewise-smooth path $\Gamma = \alpha_1 * \alpha_2 * \cdots * \alpha_b$ for some finite integer $b$ and $(\Sigma, c, F)$-admissible paths $\alpha_i$, with the following properties:
    \begin{enumerate}[label=(\arabic*)]
        \item The initial point of $\Gamma$ is $q$ and the final point of $\Gamma$ is $x^*$.
        \item If $\Gamma$ exits a given connected component of $\Sigma$, then $\Gamma$ never returns to that same component. (Formally, if $B = \Omega_i$ or $B = J_i$ for some $i$, then $\Gamma(s_1) \in B$ and $\Gamma(s_1 + \epsilon) \notin B$ for small $\epsilon > 0$ implies that $\Gamma(s) \notin B$ for all $s > s_1$.)
    \end{enumerate}
\end{lemma}

See the Appendix for a proof of Lemma~\ref{lem:path-existence-boundary}, as well as pictures.

\subsection{Proof of ABP for \texorpdfstring{$n=2$}{n=2}, general case}

\begin{proof}[Proof of Theorem~\ref{thm:ABP-n2-boundary}]

\textbf{Steps 0 and 1.}\quad
Step 0 and Step 1 are similar to the proof in the compactly supported case, but we replace $\Sigma_z$ with $f_{x_2}^{-1}(z) \cap K$.

\medskip\noindent\textbf{Step 2.}\quad
As before, let $z^* \coloneqq \sup_K f_{x_2}$, $z_* \coloneqq \inf_K f_{x_2}$, and $\Zhat \coloneqq \Zreg \cap [z_*, z^*] \setminus \set{0}$. Now $f_{x_2}$ does not generally vanish on the boundary, so
\[
    \Sigma_z = \bigsqcup_{i=1}^{m_z} \Omega_{z, i} \;\sqcup\; \bigsqcup_{i=1}^{d_z} J_{z, i},
\]
where $m_z$ and $d_z$ are positive integers, $\Omega_{z, i}$ are  simple closed curves, and $J_{z, i}$ are curves intersecting $\partial K$.

For each $z \in \Zhat$, define $S_z(x) \coloneqq f(x) - z\, p_2(x)$ for $x \in K$. Note that $S_z$ and the components $\Omega_{z, i}$ satisfy the hypotheses of Lemma~\ref{lem:ABP-1D-loop}. Let
\[
    \mathcal{I}_{z} \coloneqq \set{i \in \set{1, \ldots, d_z} \given f_{x_1}(x) = 0 \text{ for some } x \in J_{z, i}}.
\]
For $i \in \mathcal{I}_{z}$, the curve $J_{z, i}$ satisfies the hypotheses of Lemma~\ref{lem:ABP-1D-curve-vanishing}; let $q_{z, i} \in \partial K$ be the boundary point provided by that lemma. For each $i \notin \mathcal{I}_z$, let $\mathcal{P}_{z, i}$ be the set of partitions of the parameter domain for $J_{z, i}$ of the form $s_1 \geq t_1 \geq s_2 \geq t_2 \geq \cdots \geq s_k \geq t_k$, where $t_j$ are chosen such that $\hat{G}(\gamma(s_j)) = \gamma(t_j)$, as in Lemma~\ref{lem:ABP-1D-curve-nonvanishing}. For $z \in \Zhat$, set
\[
    \begin{aligned}
        \Phi_1(z) &\coloneqq \sum_{i=1}^{m_z} \Vosc{\Omega_{z, i}} S_z, \\
        \Phi_2(z) &\coloneqq \sum_{i \notin \mathcal{I}_z} \sup_{\mathcal{P}_{z, i}} \sum_{j} \bigl[S_z(\gamma(s_j)) - S_z(\gamma(t_j))\bigr], \\
        \Phi_3(z) &\coloneqq \sup_{i \in \mathcal{I}_z} \bigl[\sup_{J_{z, i}} S_z - S_z(q_{z, i})\bigr].
    \end{aligned}
\]
Define $\phi \colon \R \to [0, \infty)$ by
\[
    \phi(z) \coloneqq
    \begin{cases}
        \Phi_1(z) + \Phi_2(z) + \Phi_3(z) & \text{if } z \in \Zhat, \\[4pt]
        0 & \text{otherwise.}
    \end{cases}
\]
We claim that
\begin{equation}\label{ineq:Step2-general}
    \boxed{\phi(z) \leq \diam(K)\, \TV_z(f_{x_1}) \quad \text{for all } z \in \R.}
\end{equation}
The claim is trivial for $z \notin \Zhat$, so suppose $z \in \Zhat$. By Lemma~\ref{lem:ABP-1D-loop}, for $1 \leq i \leq m_z$,
\[
    \Vosc{\Omega_{z, i}} S_z
    \leq \Haus{1}(p_1(\Omega_{z, i})) \int_{\Omega_{z, i}} \abs*{\dv{s} f_{x_1}} \d \Haus{1}(s).
\]
Lemma~\ref{lem:ABP-1D-curve-nonvanishing} implies that for $i \notin \mathcal{I}_z$,
\[
    \sup_{\mathcal{P}_{z, i}} \sum_{j} \bigl[S_z(\gamma(s_j)) - S_z(\gamma(t_j))\bigr]
    \leq \Haus{1}(p_1(J_{z, i})) \int_{J_{z, i}} \abs*{\dv{s} f_{x_1}} \d \Haus{1}(s).
\]
Finally, for $i \in \mathcal{I}_z$, Lemma~\ref{lem:ABP-1D-curve-vanishing} ensures
\[
    \sup_{J_{z, i}} S_z - S_z(q_{z, i})
    \leq \Haus{1}(p_1(J_{z, i})) \int_{J_{z, i}} \abs*{\dv{s} f_{x_1}} \d \Haus{1}(s),
\]
and we finish by summing in $i$.

\medskip\noindent\textbf{Step 3.}\quad
We show that
\begin{equation}\label{ineq:Step3-general}
    \boxed{\bigl(\Osc{K} f - \Osc{\partial K} f\bigr)^2 \leq 16\, \diam(K) \int_{-\infty}^{\infty} \phi(z) \d z}
\end{equation}
by estimating the supremum and infimum separately, using our new definition of $\phi$ which accounts for connected components which meet $\partial K$. We only provide details for the supremum. Without loss of generality, $\sup_K f > \sup_{\partial K} f$, so there exists $x^*$ interior to $K$ such that $f(x^*) = \sup_K f$. Define $E_\lambda$ and choose $z_0 \in \Zhat$ as in Step~3 of the compactly supported case. Fix a coloring $c$ for $K \setminus \Sigma_{z_0}$ according to the sign of $z_0$ as before. Let $\Gamma$ be the path from some $q \in \partial K$ to $x^*$ provided by Lemma~\ref{lem:path-existence-boundary} applied to $F = f_{x_1}$, and let $x_0 = q,\, x_1, \ldots, x_{k-1},\, x_k = x^*$ be the endpoints of each admissible segment of $\Gamma$. Then
\[
    f(x^*) - f(q) = \sum_{j=1}^{k} \bigl[f(x_j) - f(x_{j-1})\bigr].
\]

We estimate $f(x_j) - f(x_{j-1})$ based on the type of $(\Sigma_{z_0}, c, f_{x_1})$-admissible curve lying between $x_{j-1}$ and $x_j$. By Lemma~\ref{lem:change-over-vertical-admissible-segments}, if $x_j$ and $x_{j-1}$ are endpoints of some admissible $\ell(x, y)$, then
\[
    f(x_j) - f(x_{j-1}) \leq z_0\, p_2(x_j - x_{j-1}).
\]
On the other hand, if $x_j = G(x)$ and $x_{j-1} = x$ for some $x \in \Omega_{z_0, i}$, then
\[
    f(x_j) - f(x_{j-1}) \leq \Vosc{\Omega_{z_0, i}} S_{z_0} + z_0\, p_2(x_j - x_{j-1}).
\]
Next, if $x_j \in J_{z_0, i}$ and $x_{j-1} = q_{z_0, i} \in J_{z_0, i} \cap (\partial K)$ such that $f_{x_1}(y) = 0$ for some $y \in J_{z_0, i}$ which lies between $x_j$ and $q_{z_0, i}$, then
\[
    f(x_j) - f(x_{j-1})
    \leq \sup_{i \in \mathcal{I}_{z_0}} \bigl[\sup_{J_{z_0, i}} S_{z_0} - S_{z_0}(q_{z_0, i})\bigr]
        + z_0\, p_2(x_j - x_{j-1}).
\]
Finally, if $x_j, x_{j+2}, \ldots, x_{j+2l} \in J_{z_0, i}$ and $x_{j-1} = \hat{G}(\hat{H}(x_j)),\, \ldots,\, x_{j+2l-1} = \hat{G}(\hat{H}(x_{j+2l}))$, and $f_{x_1}$ does not vanish on $J_{z_0, i}$, then
\begin{multline*}
    \sum_{t=1}^{l} \bigl[f(x_{j+2t}) - f(x_{j+2t-1})\bigr] \\
    \leq \sup_{\mathcal{P}_{z_0, i}} \sum_{j} \bigl[S_{z_0}(\gamma_i(s_j)) - S_{z_0}(\gamma_i(t_j))\bigr]
        + z_0 \sum_{t=1}^{l} p_2(x_{j+2t} - x_{j+2t-1}).
\end{multline*}

By Property~(2) of Lemma~\ref{lem:path-existence-boundary}, the path $\Gamma$ visits any given connected component at most once; hence upon summing in $i$:
\[
    f(x^*) - f(q) \leq \Phi_1(z_0) + \Phi_2(z_0) + \Phi_3(z_0) + z_0 \sum_{j=1}^{k} p_2(x_j - x_{j-1})
    = \phi(z_0) + z_0\, p_2(x^* - q).
\]
Hence the total change in $f$ over $\Gamma$ is at most $\phi(z_0) + \abs{z_0}\, \diam(K)$. Since $\abs{z_0} \leq a/\lambda$ and $\phi(z_0) < \lambda$,
\[
    \sup_K f - f(q)
    = \int_{0}^{\Haus{1}(\Gamma)} \dv{s} f(\Gamma(s)) \d \Haus{1}
    \leq (a/\lambda)\, \diam(K) + \lambda.
\]
Likewise, there exists $q' \in \partial K$ such that $f(q') - \inf_K f \leq (a/\lambda)\, \diam(K) + \lambda$, so
\[
    \Osc{K} f - \Osc{\partial K} f \leq 2(a/\lambda)\, \diam(K) + 2\lambda.
\]
We finish the argument exactly as in the compactly supported case, obtaining
\[
    \bigl(\Osc{K} f - \Osc{\partial K} f\bigr)^2 \leq C_K \int_{K} \abs{\det D^2 f(x)} \d x,
\]
where $C_K = 16\, \diam(K)^2$.
\end{proof}

\section{Discussion about ABP for \texorpdfstring{$n \geq 3$}{n>=3}}

Viterbo's proof of Theorem~\ref{thm:Viterbo-ABP} works for compactly supported $C^2$ functions in any dimension $n \geq 2$. However, our proof of Theorems~\ref{thm:ABP-n2-compact} and~\ref{thm:ABP-n2-boundary} are for $n = 2$ only. A classical analysis proof of ABP using Viterbo's approach (no contact sets, no convexity, and avoiding direct reference to symplectic geometry) may also be possible for $n \geq 3$. To see what steps may be required, for concreteness let us focus on the $n = 3$ case.

\begin{theorem}[ABP for $n=3$, compactly supported case]\label{thm:ABP-n3}
    Let $K$ be a compact subset of $\R^3$ with nonempty interior. There exists a constant $C_K$ depending on the cube of the diameter of $K$ such that, for all $f \in C^2(\R^3)$ with $\supp f \subseteq K$,
    \[
        \bigl(\Osc{K} f\bigr)^3 \leq C_K \int_{K} \abs{\det D^2 f(x)} \d x.
    \]
\end{theorem}

\subsection{Remarks about \texorpdfstring{$n=3$}{n=3}}

\noindent\textbf{Step 1.}\quad
In order to apply the co-area formula, we need an appropriate factorization of the determinant in the $n = 3$ case. One such factorization is through the triple scalar product:
\begin{equation*}\label{eq:triple-scalar-product-det}
    \det(D^2 f) = \det\bigl[\nabla f_{x_1},\, \nabla f_{x_2},\, \nabla f_{x_3}\bigr]
    = (\nabla f_{x_1} \times \nabla f_{x_2}) \cdot \nabla f_{x_3}.
\end{equation*}
Applying co-area to the mapping $f_{x_3}$ yields\footnote{We used the fact that $A_z(\nabla f_{x_1}, \nabla f_{x_2}) = A_z(\tgrad f_{x_1}, \tgrad f_{x_2})$.}
\begin{equation}\label{eq:coarea-n=3}
    \int_{K} \abs{\det D^2 f(x)} \d x
    = \int_{-\infty}^{\infty} \int_{\Sigma_z} \abs{A_z(\tgrad f_{x_1}, \tgrad f_{x_2})} \d \Haus{2} \d z,
\end{equation}
where $\Sigma_z \coloneqq f_{x_3}^{-1}(z)$, $\vec{N} \coloneqq \nabla f_{x_3} / \abs{\nabla f_{x_3}}$, $\tgrad \coloneqq (I - \vec{N} \otimes \vec{N}) \nabla$ is the gradient intrinsic to $\Sigma_z$, and $A_z(v, w) \coloneqq (v \times w) \cdot \vec{N}$ is the area form on the surface $\Sigma_z$.

An alternative factorization of $\abs{\det D^2 f}$ is based on analogy with the right-hand side of \cite[Theorem 3.10]{viterbo2000metric}. We show in the Appendix that \eqref{eq:coarea-n=3} is, at least formally, equivalent to
\begin{multline}\label{eq:coarea-n=3-alternate}
    \int_{K} \abs{\det D^2 f(x)} \d x \\
    = \int_{-\infty}^{\infty} \int_{\Sigma_z} \abs*{\det\bigl(D^2_{x' x'} f - \nabla f_{x_3} \otimes \nabla f_{x_3} / f_{x_3 x_3}\bigr)}(x', x_3(x')) \d x' \d z,
\end{multline}
where $x' \coloneqq (x_1, x_2)$ and $(x', x_3(x'))$ locally parametrizes $\Sigma_z$.

\medskip\noindent\textbf{Step 2.}\quad
For any regular value $z \in \Zhat$, by the compact support assumption there exist integers $N_z > 0$ and $m_{z,g} \geq 1$ such that
\[
    \Sigma_z = \bigsqcup_{g=0}^{N_z} \bigsqcup_{i=1}^{m_{z, g}} \Omega_{z, g, i},
\]
where each $\Omega_{z, g, i}$ is a surface of genus $g$. Let $S_z(x) \coloneqq f(x) - z\, p_3(x)$ for $x \in K$. Motivated by the $n = 2$ proof and the validity of Step~3 below, define
\[
    \phi(z) \coloneqq
    \begin{cases}
        \displaystyle \sum_{g=0}^{N_z} \sum_{i=1}^{m_{z, g}} \Vosc{\Omega_{z, g, i}} S_z & \text{if } z \in \Zhat, \\[4pt]
        0 & \text{otherwise.}
    \end{cases}
\]
Based on \cite[Theorem 3.10]{viterbo2000metric} and the formal equivalence between \eqref{eq:coarea-n=3} and \eqref{eq:coarea-n=3-alternate}, we conjecture that for all $z \in \R$,
\begin{equation}\label{ineq:Step2-n=3}
    \phi(z)^2 \lesssim \,\diam(K)^2 \int_{\Sigma_z} \abs{A_z(\tgrad f_{x_1}, \tgrad f_{x_2})} \d \Haus{2}.
\end{equation}
We currently do not have a general working proof of \eqref{ineq:Step2-n=3}; see the concluding remarks.

\medskip\noindent\textbf{Step 3.}\quad
We show that
\[\boxed{
    \bigl(\Osc{K} f\bigr)^3 \lesssim \, \diam(K) \int_{-\infty}^{\infty} \phi(z)^2 \d z.}
\]
We do have a proof of this step, which we believe is the analogue of the inequality
\begin{equation}\label{eq:induction-viterbo-step}
    \overline{d}(L)^{k+1} \leq C_{k+1} \int_{\R} \overline{d}(L_x)^k \d x,
\end{equation}
appearing in \cite[Theorem 3.8]{viterbo2000metric}. The idea is to follow the proof of Step~3 in the $n = 2$ case, but to take $a > \int_{-\infty}^{\infty} \phi(z)^2 \d z$ and replace $E_\lambda$ with the set $\set{z \in \Zhat \given \phi(z)^2 \geq \lambda}$. Since surfaces of genus $g$ are orientable, there exists a coloring of $K \setminus \Sigma_z$ that alternates between $\pm 1$ across each $\Omega_{z, g, i}$. Since each $\Omega_{z, g, i}$ is path-connected, we can then construct an admissible path $\Gamma$ from $\partial K$ to the point of maximum $x^*$ of $f$ on $K$, consisting of vertical segments and arcs of components $\Omega_{z, g, i}$ similarly to Lemma~\ref{lem:path-existence}. It follows that
\[
    \Osc{K} f \leq 2(a/\lambda)\, \diam(K) + 2\sqrt{\lambda}.
\]
Minimizing in $\lambda$ yields
\begin{equation}
    \Osc{K} f \leq 3\bigl(2 a\, \diam(K)\bigr)^{1/3},
\end{equation}
and the proof concludes in the same fashion as the $n = 2$ case.

\subsection{Concluding remarks}

Since Steps~1 and 3 go through for $n = 3$, the main obstacle to a classical analysis proof of ABP in the $n = 3$ case seems to be inequality~\eqref{ineq:Step2-n=3} of Step~2, which is a higher-dimensional version of Lemma~\ref{lem:ABP-1D-loop}. An explicit analysis proof of this step would be very interesting. However, it could be that $\phi(z)$ is not the correct quantity to consider, or perhaps such a simple adaptation of the $n = 2$ case is not possible due to the increased topological complexity arising from level sets of a function of three or more variables.

\section*{Acknowledgments}

We thank C.~Viterbo for helpful email correspondences, which assisted us in understanding his paper.

\appendix

\section{Proofs of Lemmas~\ref{lem:path-existence} and~\ref{lem:path-existence-boundary}}

\begin{proof}[Proof of Lemma~\ref{lem:path-existence}]
    Let $D_0$ be the connected component of $K \setminus \Sigma$ containing $x^*$, and let $E_0^\infty$ be the unbounded component of $\R^2 \setminus \overline{D_0}$. First we claim that there exists an admissible path $\Gamma_0(x^*)$ contained in $\overline{D_0}$ starting at some point $q_1 \in \partial(E_0^\infty) \cup \partial K$ and ending at $q_0 \coloneqq x^*$.

    For convenience, define $\alpha * \varnothing \coloneqq \varnothing * \alpha \coloneqq \alpha$, and let $\mathcal{S}$ be the class of piecewise-smooth paths in $\overline{D_0}$. Define $\Gamma_0 \colon \overline{D_0} \to \mathcal{S}$ recursively by
    \[
        \Gamma_0(x) \coloneqq
        \begin{cases}
            \varnothing & \text{if } x \in \partial(E_0^\infty) \cup \partial K, \\
            \Gamma_0(\Xi(x)) * \ell(\Xi(x), x) & \text{if } \#A(x) \geq 1, \\
            \Gamma_0(G(x)) * \gamma(G(x), x) & \text{if } \#A(x) = 0,
        \end{cases}
    \]
    where $\Xi$ is a mapping which we presently describe.

    For each $y \in \overline{D_0}$, define
    \[
        A(y) \coloneqq \set{\xi \in (\partial D_0) \cup (\partial K) \given \ell(\xi, y) \text{ is } (\Sigma, c)\text{-admissible}}.
    \]
    Now suppose $y \in \overline{D_0}$. There are three possibilities: $\#A(y) = 0$, $1$, or $2$. We observe that $\#A(y) \in \set{0, 2} \implies y \in \Sigma$. (If $y \in \overline{D_0} \setminus \Sigma$, then $c$ is constant on the segment $\ell = K \cap (\set{p_1(y)} \times (-\epsilon, \epsilon))$ for small $\epsilon > 0$, and hence there exists exactly one admissible direction to traverse $\ell$: down if $\ell \subseteq c^{-1}(+1)$, and up if $\ell \subseteq c^{-1}(-1)$. We may extend $\ell$ to the first point of intersection with $(\partial D_0) \cup (\partial K)$, showing $\#A(y) = 1$.) As a corollary, $G(y)$ is defined if $\#A(y) = 0$, in which case $\#A(G(y)) = 2$ by definition of $G$.

    If $\#A(y) = 2$, then $A(y) = \set{\xi^+, \xi^-}$ with $\xi^- \prec y \prec \xi^+$; moreover $y \in \Sigma$. Since $\partial D_0 \subseteq \Sigma$, $y \in \partial D_0$, and there exists exactly one $\xi(y) \in \set{\xi^+, \xi^-}$ such that $\ell(\xi(y), y)^\circ \subseteq D_0$. Define $\Xi \colon \set{y \in \overline{D_0} \given \#A(y) \geq 1} \to (\partial D_0) \cup (\partial K)$ by
    \[
        \Xi(y) \coloneqq
        \begin{cases}
            \text{the unique } \xi \in A(y) & \text{if } \#A(y) = 1, \\
            \xi(y) & \text{if } \#A(y) = 2.
        \end{cases}
    \]

    Now we show that $\Gamma_0(x^*)$ fulfills the requirements of the lemma. By construction, every open segment $\ell(\Xi(x), x)^\circ$ is a subset of $D_0$, and $c$ is constant on $D_0$; admissibility forces all these segments to be traversed in a \emph{uniform} direction (either all up or all down). This uniformity together with the definition of $G$ ensures that each $\Omega_i$ is visited by $\Gamma_0(x^*)$ at most once, proving Property~(2). Finally, since there are only finitely many $\Omega_i$, the recursion terminates at some point $q_1$ after finitely many steps, and by definition of $\Gamma_0$, $q_1 \in \partial(E_0^\infty) \cup \partial K$.

    If $q_1 \in \partial K$, we are done. Otherwise, let $D_1$ be the connected component of $K \setminus (\Sigma \cup \overline{D_0})$ such that $q_1 \in \overline{D_1}$. Repeat the process: there exists an admissible path $\Gamma_1(q_1)$ contained in $\overline{D_1}$ starting at some point $q_2 \in \partial(E_1^\infty) \cup \partial K$ and ending at $q_1$, where $E_1^\infty$ is the unbounded component of $\R^2 \setminus \overline{D_1}$. Continuing in this way, and using the fact that there are only finitely many $\Omega_i$, the process must eventually terminate with some admissible path $\Gamma_n(q_n)$ from $q_{n+1} \in \partial K$ to $q_n \in \overline{D_n}$. Then we take
    \[
        \Gamma \coloneqq \Gamma_n(q_n) * \Gamma_{n-1}(q_{n-1}) * \cdots * \Gamma_0(x^*).
    \]
    It is immediate that $\Gamma$ fulfills Property~(1). Property~(2) holds because it holds separately on each $\Gamma_i(q_i)$, $1 \leq i \leq n$, and by the observation that $\Gamma_i(q_i) \cap \Gamma_j(q_j) = \varnothing$ if $\abs{i - j} \geq 2$, while $\Gamma_{i+1}(q_{i+1}) \cap \Gamma_i(q_i) = \set{q_{i+1}}$.
\end{proof}
 \begin{figure}[h!]
	\centering
    \includegraphics[width=.7\textwidth, angle=0]{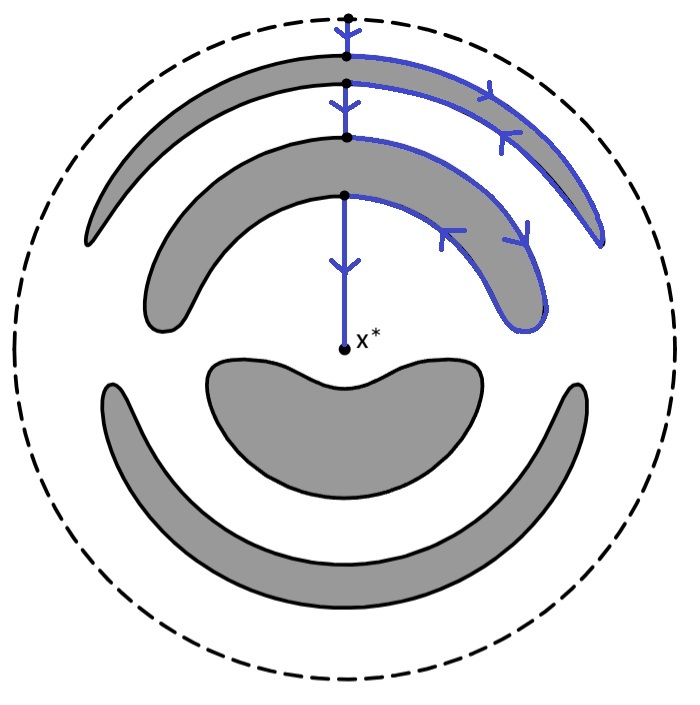}
	\caption{Example path constructed in Lemma \ref{lem:path-existence} from $\partial K$ to $x^*$ where $c = c_{0,x^*}$. The boundary of $K$ is the dashed outer circle, and the other loops are connected components of $\Sigma$. The shaded region corresponds to $c^{-1}(-1)$, and the white region corresponds to $c^{-1}(+1)$. Here, we have $\Gamma = \Gamma_0(x^*) = \ell(q_1,G(\xi_2))*\gamma(G(\xi_2), \xi_2)*\ell(\xi_2, G(\xi_1))*\gamma(G(\xi_1), \xi_1)*\ell(\xi_1,x^*)$, where  $\xi_1 = \Xi(x^*), \xi_2 = \Xi(G(\xi_1)),$ and $q_1 = \Xi(G(\xi_2)) \in \partial K$.}
\end{figure}

 \begin{figure}[h!]
	\centering
    \includegraphics[width=.7\textwidth, angle=0]{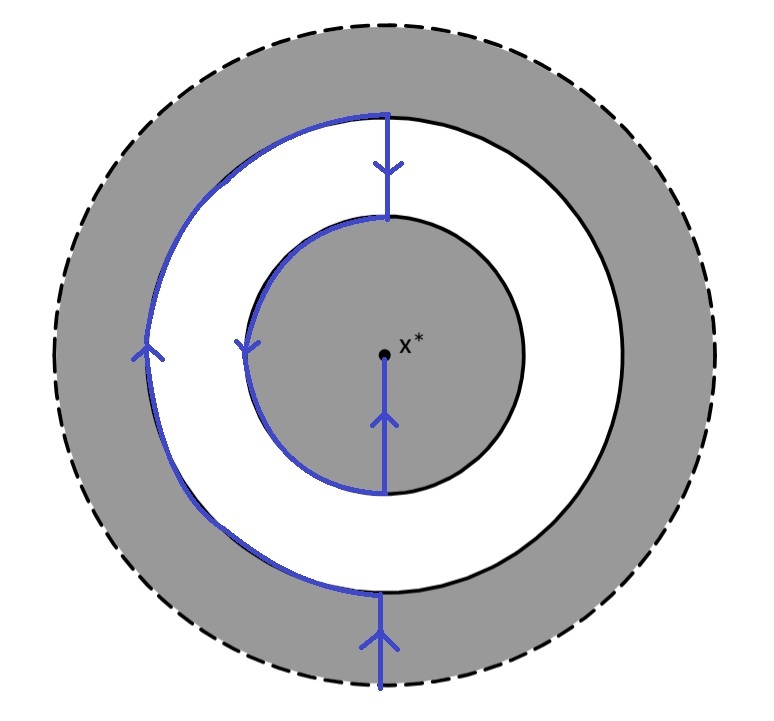}
	\caption{Example path constructed in Lemma \ref{lem:path-existence} from $\partial K$ to $x^*$ where $c = c_{1,x^*}$. In this example, we have $\Gamma = \Gamma_2(q_2)*\Gamma_1(q_1)*\Gamma_0(x^*),$ where $\Gamma_2(q_2) = \ell(q_3,G(q_2))*\gamma(G(q_2), q_2),$ $\Gamma_1(q_1) = \ell(q_2, G(q_1))*\gamma(G(q_1), q_1)$, and $\Gamma_0(x^*) = \ell(q_1,x^*)$, where  $q_i \in \partial(E_{i-1}^\infty)$ and $q_3 \in \partial K$.}
\end{figure}
\begin{proof}[Proof of Lemma~\ref{lem:path-existence-boundary}]
    The proof is similar to that of Lemma~\ref{lem:path-existence}. We begin the same way by letting $D_0$ be the connected component of $K \setminus \Sigma$ containing $x^*$. Then we construct a path $\Gamma_0(x^*)$ contained in $\overline{D_0}$ from some point $q_1 \in \partial(E_0^\infty) \cup (\partial K)$ to $x^*$. We construct this path recursively as a function $\Gamma_0 \colon \overline{D_0} \to \mathcal{S}$ defined by the following six cases:
    \begin{enumerate}[label=(\arabic*)]
        \item $\Gamma_0(x) \coloneqq \varnothing$ if $x \in \partial K$;
        \item $\Gamma_0(x) \coloneqq \Gamma_0(\Xi(x)) * \ell(\Xi(x), x)$ if $\#A(x) \geq 1$;
        \item $\Gamma_0(x) \coloneqq \Gamma_0(G(x)) * \gamma(G(x), x)$ if $\#A(x) = 0$ and $x \in \Omega_i$ for some $i$;
        \item $\Gamma_0(x) \coloneqq \Gamma_0(\hat{G}(x)) * \gamma(\hat{G}(x), x)$ if $\#A(x) = 0$ and $x \in J_i$ for some $i$;
        \item $\Gamma_0(x) \coloneqq \gamma(q, x)$ if $x \in J_i$ for some $i$ such that there exists $q \in J_i \cap (\partial K)$ with $F(y) = 0$ for some $y \in J_i$ lying between $x$ and $q$;
        \item $\Gamma_0(x) \coloneqq \Gamma_0(\hat{H}(x)) * \gamma(\hat{H}(x), x)$ if $x \in J_i$ for some $i$ such that $F$ does not vanish on $J_i$, and $\gamma(\hat{H}(x), x)$ is parametrized so that $F(\gamma(s))\, p_1(\gamma'(s)) \leq 0$.
    \end{enumerate}
    As in Lemma~\ref{lem:path-existence}, we obtain a sequence of paths $\Gamma_i(q_i)$ contained in components $\overline{D_i}$, and finally we set
    \[
        \Gamma \coloneqq \Gamma_n(q_n) * \Gamma_{n-1}(q_{n-1}) * \cdots * \Gamma_0(x^*). \qedhere
    \]

 \begin{figure}[h!]
	\centering
    \includegraphics[width=.8\textwidth, angle=0]{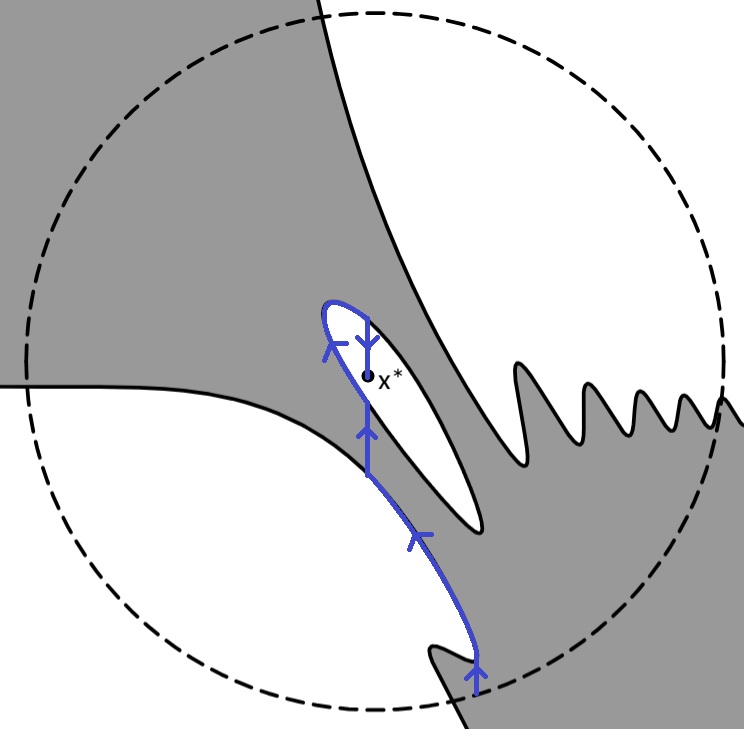}
	\caption{Example path constructed in Lemma \ref{lem:path-existence-boundary} from $\partial K$ to $x^*$ where $c = c_{0,x^*}$. The boundary of $K$ is the dashed outer circle, and the other lines are connected components of $\Sigma$ of type $\Omega_i$ or $J_i$. The shaded region corresponds to $c^{-1}(-1)$, and the white region corresponds to $c^{-1}(+1)$. In this example, we suppose that $F \neq 0$ on the lower-most curve, and that the admissible choice of parametrization is right to left. We have $\Gamma =\ell(q_3,\hat{H}(q_2) )*\gamma(\hat{H}(q_2),q_2)*\ell(q_2, G(q_1))*\gamma(G(q_1),q_1)*\ell(q_1, x^*)$. This picture was generated by plotting the zero level set of the function $2\cos(0.6(0.5 x^3 - y^2)) - (y+4)(x^2+y-4)$ and taking $x^* = (2.54, -3.62)$ and $K = \overline{B(x^*, 2.6)}$.}
\end{figure}

  \begin{figure}[h!]
	\centering
    \includegraphics[width=.8\textwidth, angle=0]{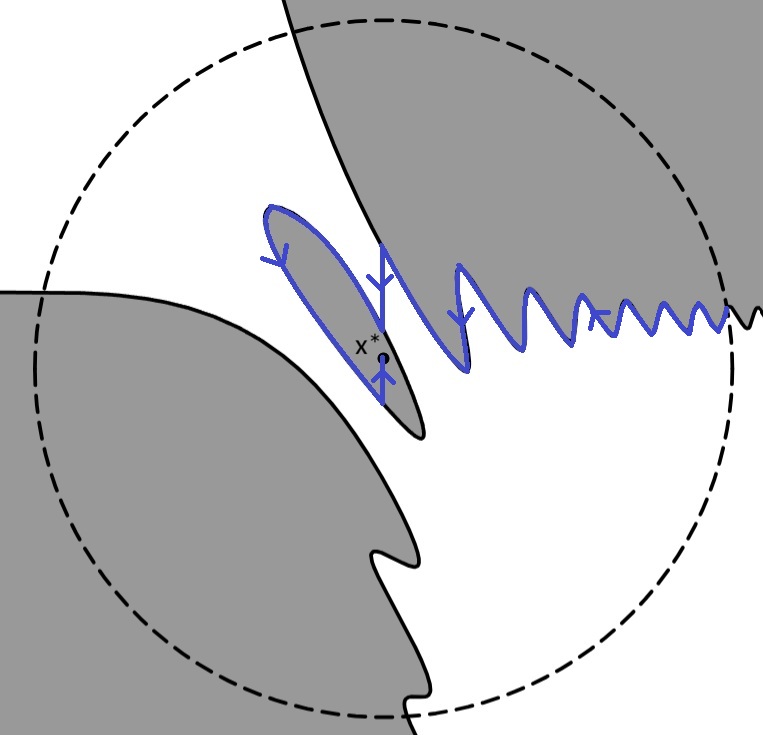}
	\caption{Example path constructed in Lemma \ref{lem:path-existence-boundary} from $\partial K$ to $x^*$ where $c = c_{1,x^*}$. In this example, we suppose that $F= 0$ at some point  on the upper curve near the right-most boundary of $K$. We have $\Gamma = \gamma(q, q_2)*\ell(q_2,G(q_1))*\gamma(G(q_1),q_1)*\ell(q_1, x^*)$ for $q \in \partial K$. This picture was generated by plotting the zero level set of the function $2\cos(0.6(0.5 x^3 - y^2)) - (y+4)(x^2+y-4)$ and taking $x^* = (3.04, -4.38)$ and $K = \overline{B(x^*, 2.6)}$.}
\end{figure}
\end{proof}
\section{Formal equivalence of \texorpdfstring{\eqref{eq:coarea-n=3}}{(coarea-n=3)} and \texorpdfstring{\eqref{eq:coarea-n=3-alternate}}{(coarea-n=3-alternate)}}

By the implicit function theorem, on any coordinate patch of $\Sigma_z$ which is a function of $x'$, we have $f_{x_3 x_3} \neq 0$ and
\[
    \d \Haus{2}(x') = \frac{\abs{\nabla f_{x_3}}}{\abs{f_{x_3 x_3}}} \d x'.
\]
Now consider a block matrix
\[
    M = \begin{bmatrix} A & b \\ b^\top & c \end{bmatrix},
\]
where $A$ is $(n-1) \times (n-1)$, $b$ is $(n-1) \times 1$, and $c$ is a scalar. If $c \neq 0$, then the Schur complement of $M$ is $A - b \otimes b / c$, which we denote $M / c$. It is well-known that $\det(M / c) = \det(M) / c$ whenever $c \neq 0$, by the one-line computation
\begin{align*}
    \det(M)
    &= \det\!\left(\begin{bmatrix} A & b \\ b^\top & c \end{bmatrix} \begin{bmatrix} I & 0 \\ -b^\top / c & 1 \end{bmatrix}\right) \\
    &= \det\!\left(\begin{bmatrix} M / c & b \\ 0 & c \end{bmatrix}\right)
    = c \cdot \det(M / c).
\end{align*}
Therefore, locally,
\begin{multline*}
    \frac{\abs{\det D^2 f}}{\abs{\nabla f_{x_3}}} \d \Haus{2}
    = \frac{\abs{\det D^2 f}}{\abs{f_{x_3 x_3}}} \d x' \\
    = \abs*{\det\bigl(D^2_{x' x'} f - \nabla_{x'} f_{x_3} \otimes \nabla_{x'} f_{x_3} / f_{x_3 x_3}\bigr)} \d x'.
\end{multline*}

\bibliographystyle{alpha}
\bibliography{sources}

\end{document}